\documentclass[11pt,reqno]{amsart}
\usepackage[nosumlimits]{mathtools}
\usepackage{amsthm}
\usepackage{graphicx, enumerate, url}
\usepackage[margin=1in]{geometry}
\usepackage{amssymb,cool}
\usepackage{algorithm}
\usepackage{algpseudocode}
\usepackage{color}
\usepackage{tikz}
\usepackage{bbm}
\usetikzlibrary{3d,calc}
\usepackage[low-sup]{subdepth}
\usepackage[normalem]{ulem}

\numberwithin{equation}{section}
\numberwithin{algorithm}{section}

\theoremstyle{plain}
\newtheorem{theorem}{Theorem}[section]
\newtheorem{proposition}[theorem]{Proposition}
\newtheorem{lemma}[theorem]{Lemma}

\theoremstyle{definition}
\newtheorem{definition}[theorem]{Definition}

\theoremstyle{remark}

\DeclareMathOperator{\diag}{diag}

\newcommand{\tp}{{\scriptscriptstyle\mathsf{T}}}

\begin{document}
\title{Geometric distance between positive definite matrices of different dimensions}
\author[L.-H.~Lim]{Lek-Heng~Lim}
\address{Computational and Applied Mathematics Initiative, Department of Statistics,
University of Chicago, Chicago, IL 60637, USA}
\email{lekheng@galton.uchicago.edu}

\author[R.~Sepulchre]{Rodolphe Sepulchre}
\address{Department of Engineering,
University of Cambridge, Cambridge, CB2 1PZ, UK}
\email{r.sepulchre@eng.cam.ac.uk}

\author[K.~Ye]{Ke~Ye}
\address{KLMM, Academy of Mathematics and Systems Science, Chinese Academy of Sciences, Beijing, 100190, China}
\email{keyk@amss.ac.cn}

\begin{abstract}
We show how the Riemannian distance on $\mathbb{S}^n_{++}$, the cone of $n\times n$ real symmetric or complex Hermitian positive definite matrices, may be used to naturally define a distance between two such matrices of different dimensions. Given that $\mathbb{S}^n_{++}$ also parameterizes $n$-dimensional ellipsoids, and inner products on $\mathbb{R}^n$,  $n \times n$ covariance matrices of nondegenerate probability distributions, this gives us a natural way to define a geometric distance between a pair of such objects of different dimensions.
\end{abstract}



\maketitle

\section{Introduction}\label{sec:intro}

It is well-known that the cone of real symmetric positive definite or complex Hermitian positive definite matrices $\mathbb{S}^n_{++}$ has a natural Riemannian metric that gives it a \emph{Riemannian distance} 
\begin{equation}\label{eq:delta2}
\delta_2: \mathbb{S}^n_{++} \times \mathbb{S}^n_{++} \to \mathbb{R}_+,\quad \delta_2 (A,B) =\Bigl[ \sum\nolimits_{j=1}^n \log^2(\lambda_j(A^{-1}B))\Bigr]^{1/2}.
\end{equation}
The Riemannian metric and distance endow $\mathbb{S}^n_{++}$ with rich geometric properties: in addition to being a Riemannian manifold, it is a symmetric space, a Bruhat--Tits space, a CAT(0) space, and a metric space of nonpositive curvature \cite[Chapter~6]{Bhatia}.

The Riemannian distance $\delta_2$  is arguably the most natural and useful distance on the positive definite cone $\mathbb{S}^n_{++}$ \cite{BonSep}. It may be thought as a generalization to $\mathbb{S}^n_{++}$ the geometric distance between two positive numbers $\lvert \log(a/b) \rvert$ \cite{BonSep}. It is invariant under any \emph{congruence} transformation of the data: $\delta_2(XAX^\tp, XBX^\tp) = \delta_2(A,B)$  for any invertible matrix $X$. Because a positive definite matrix is congruent to identity, the distance is entirely characterized by the simple formula $\delta (A,I)= \lVert \log A \rVert_F$. It is also invariant under \emph{inversion}, $\delta_2(A^{-1}, B^{-1}) = \delta_2(A,B)$, which again generalizes an important  property of the geometric distance between positive scalars, as well as any \emph{similarity} transformation: $\delta_2(XAX^{-1}, XBX^{-1}) = \delta_2(A,B)$  for any invertible matrix $X$. For comparison, all matrix norms are at best invariant under orthogonal or unitary transformations (e.g., Frobenius, spectral, nuclear, Schatten, Ky Fan norms) or otherwise only permutations and scaling (e.g., operator $p$-norms, H\"older $p$-norms, where $p \ne 2$).

From a practical perspective, $\delta_2$ underlies important  applications in computer vision \cite{Pen1},  medical imaging \cite{FleJos, MoaZer}, radar signal processing \cite{Bar}, statistical inference \cite{Pen2}, among other areas. In optimization, $\delta_2$ has been shown \cite{NesTod} to be equivalent to the metric defined by the self-concordant log barrier in semidefinite programming, i.e., $\log \det : \mathbb{S}^n_{++} \to \mathbb{R}$. In statistics, it has been shown \cite{Smith} to be equivalent to the Fisher information metric for Gaussian covariance matrix estimation problems. In numerical linear algebra, $\delta_2$ gives rise to the matrix geometric mean \cite{LawLim}, a topic that has been thoroughly studied and has many applications of its own.

We will show how $\delta_2$ naturally  gives a notion of geometric distance $\delta^+_2$ between positive definite matrices of \emph{different} dimensions, that is, we will define $\delta_2^+(A, B)$ for $A \in  \mathbb{S}^m_{++} $ and $B \in  \mathbb{S}^n_{++} $ where $m \ne n$.
Because of the ubiquity of positive definite matrices, this distance immediately extends to other objects. For example, real symmetric positive definite matrices $A \in \mathbb{S}^n_{++} $ are in one-to-one correspondence with:
\begin{enumerate}[\upshape (i)]
\item\label{it:ell} ellipsoids centered at the origin in $\mathbb{R}^n$,
\[
\mathcal{E}_A \coloneqq \{ x\in \mathbb{R} : x^\tp A x \le 1 \};
\]

\item inner products on $\mathbb{R}^n$,
\[
\langle \, \cdot, \cdot \, \rangle_A : \mathbb{R}^n \times \mathbb{R}^n \to \mathbb{R}, \quad ( x, y) \mapsto x^\tp A y;
\]

\item covariances of nondegenerate random variables $X = (X_1,\dots, X_n) : \Omega \to \mathbb{R}^n$,
\[
A = \operatorname{Cov}(X)  =  E[(X-\mu)(X-\mu)^\tp];
\]
\end{enumerate}
as well as other objects such as diffusion tensors, mean-centered Gaussians, sums-of-squares polynomials,  etc. In other words, our new notion of distance gives a way to measure separation between ellipsoids, inner products, covariances, etc, of different dimensions. Note that we may replace $\mathbb{R}$ by $\mathbb{C}$ and $x^\tp$ by $x^*$, so these results also carry over to $\mathbb{C}$.

In fact, it is easiest to describe our approach in terms of ellipsoids, by virtue of \eqref{it:ell}. The result that forms the impetus behind our distance $\delta_2^+$ is the following:
\begin{quote}
\emph{Given an $m$-dimensional ellipsoid $\mathcal{E}_A$ and an $n$-dimensional ellipsoid $\mathcal{E}_B$, say $m \le n$. The distance from $\mathcal{E}_A$ to the set of $m$-dimensional ellipsoids contained in $\mathcal{E}_B$ equals the distance from $\mathcal{E}_B$ to the set of $n$-dimensional ellipsoids containing $\mathcal{E}_A$, where both distances are measured via \eqref{eq:delta2}. Their common value gives a distance between $\mathcal{E}_A$ and $\mathcal{E}_B$ and therefore $A$ and $B$.}
\end{quote}
In addition, we show that this distance has an explicit, readily computable expression.

\subsection*{Notations and terminologies}

All results in this article will apply to $\mathbb{R}$ and $\mathbb{C}$ alike. To avoid verbosity, we adopt the convention that the term `Hermitian' will cover both `complex Hermitian' and `real symmetric.'  $\mathbb{F}$ will denote either $\mathbb{R}$ or $\mathbb{C}$. For $X \in \mathbb{F}^{m \times n}$,  $X^*$ will mean the transpose of $X$ if $\mathbb{F} = \mathbb{R}$ and the conjugate transpose of $X$ if $\mathbb{F} = \mathbb{C}$. 

We will adopt notations in \cite{Boyd2004}. Let $n$ be a positive integer. $\mathbb{S}^n$ will denote  the vector space of $n \times n$ Hermitian matrices, $\mathbb{S}^n_{+}$ the closed cone of of $n\times n$ Hermitian positive semidefinite matrices, and $\mathbb{S}^n_{++}$ the open cone of $n\times n$ Hermitian positive definite matrices. $\preceq$ will denote the partial order on $\mathbb{S}^n_{+}$ (and thus also on its subset  $\mathbb{S}^n_{++}$) defined by
\[
A\preceq B\qquad \text{if and only if}\qquad B-A \in \mathbb{S}^n_{+}.
\]
For brevity, positive (semi)definite will henceforth mean\footnote{Recall that while a complex positive (semi)definite matrix is necessarily Hermitian, a real positive (semi)definite matrix does not need to be symmetric.} Hermitian positive (semi)definite.

\section{Positive definite matrices}\label{sec:pos}

For the reader's easy reference, we will review some basic properties of positive definite matrices that we will need later: simultaneous diagonalizability, Cauchy interlacing, and majorization.

A pair of Hermitian matrices, one positive definite and the other nonsingular, may be simultaneously diagonalized. We state a version of this well-known result below \cite[Theorem~12.19]{Laub2005}.
\begin{theorem}[Simultaneous diagonalization]\label{thm:simultaneous diagonalization}
Let $A\in \mathbb{S}^n_{++}$ and $B\in \mathbb{S}^n$. Then there exists a nonsingular $X \in \mathbb{F}^{n \times n}$ such that 
\[
X A X^*  = I_n,\qquad X B X^*  = D,
\] 
where $I_n$ is the $n\times n$ identity matrix and $D$ is the diagonal matrix whose diagonal entries are eigenvalues of $A^{-1}B$.
\end{theorem}

As usual, we will order the eigenvalues of $X \in \mathbb{S}^n_{++}$ nonincreasingly:
\[
\lambda_1(X) \le \lambda_2(X) \le \cdots \le \lambda_n(X).
\]
The next two standard results may be found as \cite[Theorem~4.3.28, Corollary~7.7.4]{HJ1990}.
\begin{theorem}[Cauchy interlacing inequalities]\label{thm:interlacing}
Let $m \le n$ and $A \in \mathbb{S}^{n}$. If we partition $A$ into
\begin{gather*}
A = \begin{bmatrix}
A_1 & A_2 \\
A^*_2 & A_3
\end{bmatrix}, \quad A_1 \in \mathbb{S}^m, \quad A_2 \in \mathbb{F}^{m \times (n-m)}, \quad A_3\in \mathbb{S}^{n-m},
\shortintertext{then}
\lambda_j (A) \le \lambda_j (A_1) \le \lambda_{j + n -m} (A),\quad j=1,\dots, m.
\end{gather*}
\end{theorem}

\begin{proposition}[Majorization]\label{prop:partial order v.s. eigenvalues}
If $A, B \in \mathbb{S}^n_{++}$ and $A \preceq B$, then $\lambda_j(A) \le \lambda_j(B)$, $ j=1,\dots, n$.
\end{proposition}

\section{Containment of ellipsoids of different dimensions}

It helps to picture our construction with a concrete geometric object in mind and for this purpose we will exploit the one-to-one correspondence between positive definite matrices and ellipsoids mentioned in Section~\ref{sec:intro}. For $A\in \mathbb{S}^n_{++}$, the $n$-dimensional \emph{ellipsoid} $\mathcal{E}_A$ centered at the origin is
\[
\mathcal{E}_A \coloneqq \lbrace x\in \mathbb{F}^n: x^* Ax \le 1 \rbrace.
\]
All ellipsoids in this article will be centered at the origin and henceforth we will drop the `centered at the origin' for brevity.
There is a simple equivalence between containment of ellipsoids and the partial order on positive definite matrices.
\begin{lemma}\label{lem:ellipsoid-matrix1}
Let $A,B\in \mathbb{S}^n_{++}$. Then $\mathcal{E}_A \subseteq \mathcal{E}_B$ if and only if $B\preceq A$.
\end{lemma}
\begin{proof}
If $\mathcal{E}_A \subseteq \mathcal{E}_B$, then for each $x\in \mathbb{F}^n$ satisfying 
\begin{equation}\label{eq:ellipsoid-matrix1}
x^* Ax \le 1
\end{equation}
we also have $x^* Bx  \le 1$. Thus we have $y^* B y  \le y^* A y$ for any $y\in \mathbb{F}^n$ since $x = y/\sqrt{y^* Ay}$ satisfies \eqref{eq:ellipsoid-matrix1}. Conversely, if $B\preceq A$, then whenever $x$ satisfies \eqref{eq:ellipsoid-matrix1}, we have 
$x^* Bx  \le x^* Ax  \le 1$.
\end{proof}  
Lemma~\ref{lem:ellipsoid-matrix1} gives the one-to-one correspondence we have alluded to: $\mathcal{E}_A = \mathcal{E}_B$ if and only if $A = B \in \mathbb{S}^n_{++}$.

We extend this to the containment of ellipsoids of different dimensions.
Let $m\le n$ be positive integers and $A \in \mathbb{S}^m_{++}$, $B \in \mathbb{S}^n_{++}$. Consider the embedding
\[
\iota_{m,n}: \mathbb{F}^m \to \mathbb{F}^n,\quad (x_1,\dots, x_m) \mapsto (x_1,\dots, x_m,0,\dots, 0).
\]
Then we have 
\[
\iota_{m,n} (\mathcal{E}_A) = \lbrace (x,0) \in \mathbb{F}^n:  x^* A x \le 1 \rbrace,
\]
where $x \in \mathbb{F}^m$ and $0 \in \mathbb{F}^{n -m}$ is the zero vector.
Let $B_{11}$ be the upper left $m\times m$ principal submatrix of $B\in \mathbb{S}^n_{++}$, i.e., $B = \begin{bsmallmatrix}
B_{11} & B_{12} \\
B_{12}^*  & B_{22}
\end{bsmallmatrix}$ for matrices $B_{11},B_{12}, B_{22}$ of appropriate dimensions. Then the same argument used in the proof of Lemma~\ref{lem:ellipsoid-matrix1} gives the following.
\begin{lemma}\label{lem:ellipsoids-matrix2}
Let $m\le n$  and $A \in \mathbb{S}^m_{++}$, $B \in \mathbb{S}^n_{++}$. Then
$\iota_{m,n}(\mathcal{E}_A) \subseteq \mathcal{E}_B$ if and only if $B_{11} \preceq A$.
\end{lemma}

\section{Geometric distance between ellipsoids of different dimensions}

Our method of defining a geometric distance $\delta_2^+$ for pairs of positive definite matrices of different dimensions is inspired by a similar (at least in spirit) extension of the  distance on a Grassmannian to subspaces of different dimensions proposed in \cite{schubert}. The following convex sets will play the role of the Schubert varieties in \cite{schubert}.

\begin{definition}
Let $m \le n$. For any $A\in \mathbb{S}^m_{++}$, we define the \emph{convex set of $n$-dimensional ellipsoids containing $\mathcal{E}_A$} to be
\begin{equation}\label{eq:omega+}
\Omega_+ (A) \coloneqq \biggl\{ G = \begin{bmatrix}
G_{11} & G_{12} \\
G_{12}^*  & G_{22}
\end{bmatrix}\in \mathbb{S}^n_{++} : G_{11} \preceq A  \biggr\}.
\end{equation}
For any $B\in \mathbb{S}^n_{++}$, we define the \emph{convex set of $m$-dimensional ellipsoids contained in $\mathcal{E}_B$} to be
\begin{equation}\label{eq:omega-}
\Omega_{-} (B) \coloneqq \{ H\in \mathbb{S}^m_{++}: B_{11} \preceq H \},
\end{equation}
where $B_{11}$ is the upper left $m \times m$ principal submatrix of $B$. 
\end{definition}

Lemma~\ref{lem:ellipsoids-matrix2} provides justification for the names: more precisely, $\Omega_+(A)$ parametrizes all $n$-dimensional ellipsoids containing $\iota_{m,n} (\mathcal{E}_A)$ whereas  $\Omega_{-}(B) $ parametrizes all $m$-dimensional ellipsoids contained in $\mathcal{E}_{B_{11}}$. 

Given  $A \in \mathbb{S}^m_{++}$ and $B \in \mathbb{S}^n_{++}$, a natural way to  define the distance between $A$ and $B$ is to define it as the distance from  $A$ to the set $\Omega_-(B)$, i.e., 
\begin{equation}\label{eq:dAB-}
\delta_2 \bigl(A,\Omega_-(B)\bigr) \coloneqq \inf_{H\in \Omega_-(B)}  \delta_2(A,H) = \inf_{H\in \Omega_-(B)}  \Bigl[ \sum\nolimits_{j=1}^m \log^2 \lambda_j (AH^{-1})\Bigr]^{1/2};
\end{equation}
but another equally natural way is to define it as the distance from  $B\in \mathbb{S}^n_{++}$ to the set $\Omega_+(A)$, i.e.,
\begin{equation}\label{eq:dAB+}
\delta_2 \bigl(B, \Omega_+(A)\bigr) \coloneqq \inf_{G\in \Omega_+(A)}  \delta_2(G,B) =\inf_{G\in \Omega_+(A)} \Bigl[ \sum\nolimits_{j=1}^n \log^2 \lambda_j(G B^{-1})\Bigr]^{1/2}.
\end{equation}
We will show that
\[
\delta_2 \bigl(A,\Omega_-(B)\bigr) = \delta_2 \bigl(B, \Omega_+(A)\bigr)  
\]
and their common value gives the distance we seek between $A$ and $B$.

Note that $\Omega_+ (A)  \subseteq \mathbb{S}^n_{++}$ and $\Omega_{-} (B) \subseteq \mathbb{S}^m_{++}$,  \eqref{eq:dAB-} is the distance of a point $A$ to a set $\Omega_-(B)$ within the Riemannian manifold $\mathbb{S}^m_{++}$, \eqref{eq:dAB+} is the distance of a point $B$ to a set $\Omega_+(A)$ within the Riemannian manifold $\mathbb{S}^n_{++}$. There is no reason to expect that they are equal but in fact they are --- this is our main result.

\begin{theorem}\label{thm:delta}
Let $m\le n$ be positive integers and let $A \in \mathbb{S}^m_{++}$ and $B \in \mathbb{S}^n_{++}$. Let $B_{11}$ be the upper left $m\times m$ principal submatrix of $B$. Then
\begin{equation}\label{eq:equal}
\delta_2 \bigl(A,\Omega_-(B)\bigr) = \delta_2 \bigl(B, \Omega_+(A)\bigr) 
\end{equation}
and their common value is given by
\begin{equation}\label{eq:delta}
\delta^+_2(A, B) \coloneqq \Bigl[\sum\nolimits_{j=1}^m \min\{0,\log \lambda_j(A^{-1}B_{11})\}^2\Bigr]^{1/2},
\end{equation}
or, alternatively,
\[
\delta^+_2(A, B) =\Bigl[  \sum\nolimits_{j=1}^k \log^2 \lambda_j (A^{-1}B_{11})\Bigr]^{1/2},
\]
where $k$ is such that $\lambda_j (A^{-1}B_{11}) \le 1$ for $j = k+1,\dots,m$.
\end{theorem}
We will defer the proof of Theorem~\ref{thm:delta}  to Section~\ref{sec:proof} but first make a few immediate observations regarding this new distance.

An implicit assumption in Theorem~\ref{thm:delta} is that whenever we write $\delta^+(A,B)$, we will require that the dimension of the matrix in the first argument be not more than the dimension of the matrix in the second argument. In particular, $\delta^+(A,B) \ne \delta^+(B, A)$; in fact the latter is not meaningful except in the case when $m = n$. An immediate conclusion is that $\delta_2^+$ does not define a \emph{metric} on $\bigcup_{n=1}^\infty \mathbb{S}^n_{++}$, which is not surprising as $\delta_2^+$ is a distance in the sense of a distance from a point to a set.

For the special case $m = n$, \eqref{eq:delta} becomes
\[
\delta^+_2(A, B) =\Bigl[\sum\nolimits_{j=1}^m \min\{0,\log \lambda_j(A^{-1}B)\}^2\Bigr]^{1/2}.
\]
However, since $m = n$, we may swap the matrices $A$ and $B$ in \eqref{eq:equal} to get
\[
\delta_2\bigl(B, \Omega_-(A)\bigr) = \delta_2\bigl(A,\Omega_+(B)\bigr)
\]
and their common value is given by
\[
\delta_2^{+}(B, A) = \Bigl[\sum\nolimits_{j=1}^m \min\{0,\log \lambda_j(B^{-1}A)\}^2\Bigr]^{1/2}.
\]
Note that even in this case, $\delta^+(A,B) \ne \delta^+(B, A)$ in general. Nevertheless, this gives us the relation between our original Riemannian distance $\delta_2$ and the distance $\delta_2^+$ defined in Theorem~\ref{thm:delta}.
\begin{proposition}\label{prop:equal}
Let $m = n$. Then the distances $\delta_2$ in \eqref{eq:delta2} and $\delta^+_2$ in \eqref{eq:delta} are related via
\[
\delta_2(A,B) = \delta_2^{+}(A,B) + \delta_2^+(B,A).
\]
\end{proposition}

The domain of $\delta_2^+$ may be further extended to positive semidefinite matrices in the following sense: Suppose $A \in \mathbb{S}^m_+$ and $B \in \mathbb{S}^n_+$ with $m \le n$. We may replace $\mathbb{S}^m_{++}$ by $\mathbb{S}^m_+$ in the \eqref{eq:omega+} and $\mathbb{S}^n_{++}$ by $\mathbb{S}^n_+$ in \eqref{eq:omega-}. If $A$ is singular, i.e., it is positive semidefinite but not positive definite, then we have 
\begin{equation}\label{eq:Asing}
\delta_2\bigl(A,\Omega_-(B)\bigr)  =\infty = \delta_2\bigl(B, \Omega_+(A)\bigr).
\end{equation}
as $\delta_2(A,H) = \infty$ for any $H\in \Omega_-(B)$ and $\delta_2(B,G) = \infty$ for any $G\in \Omega_+(A)$.
However, if $B$ is singular, then \eqref{eq:Asing} is not true unless $A$ is also singular. In general we only have
\[
\delta_2\bigl(A,\Omega_-(B)\bigr) \le \delta_2\bigl(B, \Omega_+(A)\bigr) = \infty,
\]
where the inequality can be strict  when $A$ is positive definite. In short, \eqref{eq:equal} extends to positive semidefinite $A$ and $B$ except in the case where $A$ is nonsingular and $B$ is singular.

\section{Proof of Theorem~\ref{thm:delta}}\label{sec:proof}

Throughout this section, we will assume that $m \le n$, $A \in \mathbb{S}^m_{++}$, and $B \in \mathbb{S}^n_{++}$. We will prove Theorem~\ref{thm:delta} by showing that
\begin{equation}\label{eq:delta-}
\delta_2\bigl(A,\Omega_-(B)\bigr) = \Bigl[\sum\nolimits_{j=1}^m \min\{0,\log \lambda_j(A^{-1}B_{11})\}^2\Bigr]^{1/2}
\end{equation}
in Lemma~\ref{lem:delta-} and 
\begin{equation}\label{eq:delta+}
\delta_2\bigl(B,\Omega_+(A)\bigr) = \Bigl[\sum\nolimits_{j=1}^m \min\{0,\log \lambda_j(A^{-1}B_{11})\}^2\Bigr]^{1/2}
\end{equation}
in Lemma~\ref{lem:delta+}. The key to establishing these is to repeatedly use the following invariance of $\delta_2$ under congruence action by nonsingular matrices.
\begin{lemma}[Invariance of $\delta_2$]\label{lem:invariance}
Let $A,B\in \mathbb{S}^n_{++}$ and $X\in \mathbb{F}^{n\times n}$ be nonsingular. Then 
\[
\delta_2(X A X^* ,X B X^* ) = \delta_2 (A,B).
\]
\end{lemma}
\begin{proof}
Observe that 
\[
(X A X^{* } )(X B X^{* })^{-1} = X (A B^{-1}) X^{-1}.
\]
Thus $\lambda_j (A B^{-1}) = \lambda_j ( (X A X^{* } )(X B X^{* })^{-1} )$ and the invariance of $\delta_2$ follows.
\end{proof}

\subsection{Calculating $\delta_2\bigl(A,\Omega_-(B)\bigr)$}\label{sec:delta-}
Recall that we partition  $B \in \mathbb{S}^n_{++}$ into $B = \begin{bsmallmatrix}
B_{11} & B_{12} \\
B_{12}^*  & B_{22}
\end{bsmallmatrix}$. Note that $B_{11}\in \mathbb{S}^m_{++}$, $B_{12} \in \mathbb{F}^{m \times (n-m)}$, and $B_{22} \in \mathbb{S}^{n-m}_{++}$. By Theorem~\ref{thm:simultaneous diagonalization}, there is a nonsingular $X\in \mathbb{F}^{m \times m}$ such that 
\[
X A X^*  = I_m,\qquad X B_{11} X^*  = D,
\]
where $D = \diag(\lambda_1,\dots,\lambda_m)$ with  $\lambda_j \coloneqq \lambda_j(A^{-1} B_{11})$, $j =1,\dots, m$. Since $B$ is positive definite, so is $B_{22}$, and thus there is a nonsingular $Y\in \mathbb{F}^{(n-m) \times (n-m)}$ such that 
\[
Y B_{22} Y^*  = I_{n-m}.
\]
Therefore, we have 
\[
\begin{bmatrix}
X & 0 \\
0 & Y
\end{bmatrix} \begin{bmatrix}
B_{11} & B_{12} \\ 
B_{12}^*   & B_{22}
\end{bmatrix} \begin{bmatrix}
X^*  & 0 \\
0 & Y^* 
\end{bmatrix} = \begin{bmatrix}
D & X B_{12} Y^*  \\
Y B_{12}^*  X^*  & I_{n-m}
\end{bmatrix}.
\]
Set $Z \coloneqq \begin{bsmallmatrix}
X & 0 \\
0 & Y
\end{bsmallmatrix}$. Then, by Lemma~\ref{lem:invariance},
\[
\delta_2 \bigl(A, \Omega_-(B)\bigr) = \delta_2 \bigl(X A X^* , X \Omega_-(B) X^* \bigr) = \delta_2 \bigl(I_m, \Omega_-(Z B Z^* )\bigr).
\]
Hence we may assume without loss of generality that  
\begin{equation}\label{eq:reduction of delta-2}
A = I_m,\qquad  B = \begin{bmatrix}
D & B_{12} \\
B_{12}^*  & I_{n-m}
\end{bmatrix},
\end{equation}
where  $D = \diag(\lambda_1,\dots,\lambda_m)$ and $B_{12}\in \mathbb{F}^{m \times (n-m)}$ is such that $B$ is positive definite.

We will need a small observation regarding the eigenvalues of $B$.
\begin{lemma}
Let $D = \diag(\lambda_1,\dots,\lambda_m)$. Let $\mu_{m+1},\dots, \mu_{n}$ be the eigenvalues of $B_{12}^*  D^{-1} B_{12}$. Then $0< \mu_{m+j} < 1$ for all $j=1,\dots,n-m$ and the eigenvalues of $B = \begin{bsmallmatrix}
D & B_{12} \\
B_{12}^*  & I_{n-m}
\end{bsmallmatrix}$ are $\lambda_1,\dots, \lambda_m,1- \mu_{m+1},\dots, 1-\mu_n$. 
\end{lemma}
\begin{proof}
Since $I_{n-m} - B_{12}^*  D^{-1} B_{12}$ is the Schur complement of $D$ in the positive definite matrix $B$, it follows that $0 < \mu_{m+j} < 1$ for all $j=1,\dots, n-m$. The eigenvalues of $B$ are obvious from
\[
\begin{bmatrix}
I_m & 0 \\
-B_{12}^* D^{-1} & I_{m-n}
\end{bmatrix}   \begin{bmatrix}
D & B_{12} \\
B_{12}^*  & I_{n-m}
\end{bmatrix} 
\begin{bmatrix}
I_m & 0 \\
-B_{12}^* D^{-1} & I_{m-n}
\end{bmatrix}^{-1} = \begin{bmatrix}
D & 0 \\
0 &  I_{n-m} - B_{12}^*  D^{-1} B_{12} 
\end{bmatrix}.
\]\par \vspace{-1.1\baselineskip}
\qedhere
\end{proof}

We are now ready to prove  \eqref{eq:delta-}.
\begin{lemma}\label{lem:delta-}
Let $m\le n$ be positive integers and let $A \in \mathbb{S}^m_{++}$ and $B \in \mathbb{S}^n_{++}$. 
Then there exists an $H_0\in \mathbb{S}^m_{++}$ such that
\[
\delta_2\bigl(A,\Omega_-(B)\bigr) = \delta_2(A,H_0) = \Bigl[ \sum\nolimits_{j=1}^m \min\{0,\log \lambda_j \}^2 \Bigr]^{1/2}.
\]
\end{lemma}
\begin{proof}
By the preceding discussions, we may assume that $A$ and $B$ are as in \eqref{eq:reduction of delta-2}. So we must have 
\[\delta_2\bigl(A,\Omega_-(B)\bigr) =\inf_{D\preceq H}  \Bigl[ \sum\nolimits_{j=1}^m \log^2 \lambda_j (H)\Bigr]^{1/2}.
\]
The condition $D\preceq H$ implies that $\lambda_j \le \lambda_j (H)$,  $j =1,\dots, m$, by Proposition~\ref{prop:partial order v.s. eigenvalues}. Hence
\begin{equation}\label{eq:delta-}
\inf_{D\preceq H} \log^2  \lambda_j(H) = \begin{cases}
\log^2 \lambda_j &\text{if}\; \lambda_j > 1,\\
0 &\text{if}\; \lambda_j \le 1.
\end{cases}
\end{equation}
Let $H_0  =\diag( h_1,\dots, h_m)$ where
\[
h_j = \begin{cases}
\lambda_j &\text{if}\; \lambda_j > 1,\\
1 & \text{if}\; \lambda_j \le 1.
\end{cases}
\]
Then it is clear that  $D \preceq H_0$  and $H_0$ is our desired matrix by \eqref{eq:delta-}.
\end{proof}

\subsection{Calculating $\delta_2\bigl(B,\Omega_+(A)\bigr)$}\label{sec:delta+}

Let $A\in \mathbb{S}^m_{++}$ and $B\in \mathbb{S}^n_{++}$. Again, we partition $B$ as in Section~\ref{sec:delta-}. Let $L$ be the upper triangular matrix
\[
L = \begin{bmatrix}
I_m & 0 \\
-B_{12}^* B_{11}^{-1}  & I_{n-m}
\end{bmatrix}.
\]
Then
\[
L B L^*  = \begin{bmatrix}
B_{11} & 0 \\
0 & I_{n-m} - B_{12}^*  B_{11}^{-1} B_{12}
\end{bmatrix}\qquad \text{and}\qquad L \Omega_+(A) L^*  = \Omega_+(A).
\]
For the second equality, observe that $L \Omega_+ (A) L^*  \subseteq \Omega_+ (A)$ and check that $L^{-1} \Omega_+ (A) (L^{-1})^*  \subseteq \Omega_+ (A)$, which implies that $\Omega_+(A) \subseteq L \Omega_+ (A) L^* $. Therefore, by Lemma~\ref{lem:invariance}, we have
\begin{equation}\label{eq:reduction of delta+1}
\delta_2\bigl(B,\Omega_+(A)\bigr) = \delta_2\bigl(L B L^*, L \Omega_+(A) L^* \bigr) = \delta_2\bigl( LBL^*, \Omega_+(A) \bigr).
\end{equation}

Let $X_1\in \mathbb{F}^{m\times m}$ and $Y_1\in \mathbb{F}^{(n-m)\times (n-m)}$ be nonsingular matrices\footnote{We may take $X_1 = D^{-1/2}X$ where $X$ and $D$ are as in the beginning of Section~\ref{sec:delta-}. $X_1$ exists by Theorem~\ref{thm:simultaneous diagonalization} and $Y_1$ exists as $I_{n-m} - B_{12}^*  B_{11}^{-1} B_{12}$ is the Schur complement of $B_{11}$ in $B$, which is positive definite.} such that 
\[
X_1 A X_1^*  = D^{-1},\quad X_1 B_{11} X_{1}^{* } = I_{n-m}, \quad Y_1 (I_{n-m} - B_{12}^*  B_{11}^{-1} B_{12}) Y_1^*  = I_{n-m},
\]
where  $D = \diag(\lambda_1,\dots,\lambda_m)$ with  $\lambda_j \coloneqq \lambda_j(A^{-1} B_{11})$, $j =1,\dots, m$. Let $Z_1 = \begin{bsmallmatrix}
X_1 & 0 \\
0 & Y_1
\end{bsmallmatrix}$. Then
\[
Z_1 L B L^*  Z_1^*  = I_n \qquad \text{and}\qquad Z_1\Omega_+(A) Z_1^*  = \Omega_+(D^{-1}).
\]
Hence, by \eqref{eq:reduction of delta+1} and Lemma~\ref{lem:invariance},
\[
\delta_2\bigl(B,\Omega_+(A)\bigr) = \delta_2\bigl(LBL^*, \Omega_+(A)\bigr) = \delta_2\bigl( Z_1LBL^*  Z_1^*, Z_1 \Omega_+(A) Z_1^* \bigr) = \delta_2\bigl(I_n, \Omega_+(D^{-1})\bigr),
\]
So to calculate $\delta_2\bigl(B,\Omega_+(A)\bigr)$, it suffices to assume that
\begin{equation}\label{eq:reduction of delta+2}
A = D^{-1} = \diag(\lambda_1^{-1},\dots, \lambda_m^{-1}),\qquad B = I_n.
\end{equation}

We are now ready to prove  \eqref{eq:delta+}.
\begin{lemma}\label{lem:delta+}
Let $m\le n$ be positive integers and let $A \in \mathbb{S}^m_{++}$ and $B \in \mathbb{S}^n_{++}$. 
Then there exists some $G_0\in \mathbb{S}^n_{++}$ such that 
\[
\delta_2\bigl(B,\Omega_+(A)\bigr)= \delta_2(G_0,B) = \Bigl[\sum\nolimits_{j=1}^m  \min\{0,\log \lambda_j(A^{-1}B_{11})\}^2 \Bigr]^{1/2}.
\]
\end{lemma}
\begin{proof}
By the preceding discussions, we may assume that $A$ and $B$ are as in \eqref{eq:reduction of delta+2}. 
So we must have 
\[
\delta_2\bigl(I_n, \Omega_+(D^{-1})\bigr) =\inf_{G_{11} \preceq D^{-1}} \Bigl[\sum\nolimits_{j=1}^n \log^2 \lambda_j(G)\Bigr]^{1/2},
\]
where $G_{11}$ is the upper left $m\times m$ principal submatrix of $G\in \Omega_+(D^{-1})$. By Proposition~\ref{prop:partial order v.s. eigenvalues}, we have $\lambda_j(G_{11}) \le \lambda_j^{-1}$, $j=1,\dots,m$. Moreover, by Theorem~\ref{thm:interlacing},
\[
\lambda_j(G) \le \lambda_j(G_{11}) \le \lambda_j^{-1},\quad j=1,\dots, m.
\]
Therefore, for each $j=1,\dots, m$,
\begin{equation}\label{eq:delta+}
\inf_{G_{11} \preceq D^{-1}} \log^2 \lambda_j(G) = \begin{cases}
\log^2 \lambda_j & \text{if}\; \lambda_j > 1, \\
0 & \text{if}\; \lambda_j \le 1,
\end{cases}
\end{equation}
and for each $j=m+1,\dots, n$,
\[
\inf_{G_{11} \preceq D^{-1}} \log^2 \lambda_j (G) = 0.
\]
Let $G_0 = \diag(g_1,\dots, g_n)$ where 
\[
g_j = \begin{cases}
\lambda_j^{-1} &\text{if}\; \lambda_j > 1 \; \text{and}\; j = 1,\dots,m, \\
1 & \text{otherwise}.
\end{cases}
\]
Then it is clear that $(G_0)_{11} \preceq D^{-1}$  and  $G_0$ is our desired matrix by \eqref{eq:delta+}.
\end{proof}

\section*{Acknowledgment}

This work came from an impromptu discussion of the first two authors at the workshop  on ``Nonlinear Data: Theory and Algorithms,'' where they met for the first time. LHL and RS gratefully acknowledge the workshop organizers and the  Mathematical Research Institute of Oberwolfach  for hosting the wonderful event.

The research leading to these results has received funding from the Defense Advanced Research Projects Agency under the DARPA Grant D15AP00109, the European Research Council under the Advanced ERC Grant Agreement Switchlet n.670645, the National Natural Science Foundation of China under the NSFC Grant no.~11688101, and the National Science Foundation under the NSF Grant IIS 1546413. In addition, LHL's work is  supported by a DARPA Director's Fellowship and  the Eckhardt Faculty Fund from the University of Chicago; KY's work is supported by the Hundred Talents Program of the Chinese Academy of Sciences as well as  the Thousand Talents Plan of the State Council of China.

\bibliographystyle{abbrv}

\end{document}